\documentclass[a4paper]{amsart}
\usepackage{amsmath, amssymb, amsthm}
\usepackage{braket}
\usepackage{url}
\usepackage{listings}
\usepackage{bm}
\usepackage{graphicx}

\theoremstyle{definition}  

\newtheorem{thm}{Theorem}[section]

\newtheorem{Theorem}[thm]{Theorem}

\newtheorem{Proposition}[thm]{Proposition}

\theoremstyle{definition}  

\newtheorem{Remark}[thm]{Remark}

\newtheorem*{acknowledgment}{Acknowledgment}

\theoremstyle{definition}  

\newcommand{\RR}{\mathbb{R}}

\newcommand{\PPP}{\mathcal{P}}

\newcommand{\zz}{\boldsymbol{z}}

\newcommand{\zero}{\boldsymbol{0}}

\newcommand{\Ker}{\operatorname{Ker}}

\newcommand{\Ann}{\operatorname{Ann}}

\makeatletter
\DeclareSymbolFont{symbolsC}{U}{txsyc}{m}{n}
\DeclareMathSymbol{\MYPerp}{\mathrel}{symbolsC}{121}
\makeatother

\newcommand{\der}{\partial}

\allowdisplaybreaks[3] 

\newbox{\MyTrashBox}
 
\begin{document}

\title{On Artinian Gorenstein algebras associated to the face posets of regular polyhedra}
\author[A. Yazawa]{Akiko Yazawa}
\address[Akiko Yazawa]{Department of Science and Technology,
	Graduate School of Medicine, Science and Technology,
	Shinshu University,
	Matsumoto, Nagano, 390-8621, Japan}
\email{yazawa@math.shinshu-u.ac.jp}

\date{}
\maketitle


\begin{abstract} 
We introduce Artinian Gorenstein algebras defined by the face posets of regular polyhedra. 
We consider the strong Lefschetz property and Hodge--Riemann relation for the algebras. 
We show the strong Lefschetz property of the algebras for all Platonic solids. 
On the other hand, for some Platonic solids, we show that the algebras do not satisfy the Hodge--Riemann relation with respect to some strong Lefschetz elements.  
\end{abstract}


\section{Introduction}\label{Introduction}

A matroid is a simplicial complex with the independence augmentation property. 
Facets of a matroid are called bases for the matroid. 
They satisfy a property of ``symmetricity'', so called the basis exchange property. 
For a matroid $M$, define the polynomial 
\begin{align*}
F_{M}=\sum_{F}\prod_{i\in F}x_{i}
\end{align*}
where $F$ runs the collection of facets of $M$. 
The polynomials are studied from various viewpoints. 
In \cite{ALOV2018, MR4003314, MR3899575, BH2018, MR4172622}, 
it was shown that the polynomials are log-concave on the positive orthant. 
For a polynomial $F$, to study log-concavity and to study the Hessian matrix are equivalent. 
More precisely, $F$ is log-concave on the positive orthant if and only if the Hessian matrix $H_{F}(\bm{a})$ of $F$ has exactly one positive eigenvalue for $\bm{a}$ in the positive orthant. 
In \cite{MNY2020}, it was shown that the polynomials are strictly log-concave on the positive orthant, 
equivalently, the Hessian matrices have exactly one positive eigenvalue and are not degenerate. 
For an Artinian Gorenstein algebra associated to the polynomial, it was also shown that the strong Lefschetz property and Hodge--Riemann relation at degree one.  

We introduce the homogeneous polynomial $F_{\PPP}$ for a regular polyhedron $\PPP$ as an analogue of above polynomials. 
For a regular polyhedra $\PPP$, we define a homogeneous polynomial by
\begin{align*}
F_{\PPP}=\sum_{F\in F(\PPP)}\prod_{v\in F}x_{v}, 
\end{align*}
where $F(\PPP)$ is the collection of the facets, and $v$ is a $0$-dimensional face. 
Then we consider the Artinian Gorenstein algebra 
\begin{align*}
A_{\PPP}=\RR[\der_{v}|v\in V(\PPP)]/\Ann(F_{\PPP}), 
\end{align*}
where $\der_{v}$ is the partial derivative operator of $x_{v}$, and $V(\PPP)$ is the collection of $0$-dimensional faces.  
We discuss the strong Lefschetz property of $A_{\PPP}$ and Hodge--Riemann relation with respect to the Poincar\'e duality 
\begin{align*}
P_{F_{\PPP}}^{k}: A_{k}\times A_{s-k}\to \RR, && (f,g)\mapsto fgF_{\PPP}. 
\end{align*}

In this paper, we consider the Platonic solids. 
And, we show the following, and study the strong Lefschetz elements. 
\begin{Theorem}[cf.\ Theorems \ref{n=4}, \ref{n=6}, \ref{n=8}, \ref{n=12} and \ref{n=20}]
The following hold for the Platonic solids: 
\begin{itemize}
\item The Artinian Gorenstein algebras have the strong Lefschetz property. 
\item For the regular tetrahedron and octahedron, the algebra satisfy the Hodge--Riemann relation on the positive orthant. 
For the others, the algebras do not satisfy the Hodge--Riemann relation with respect to some strong Lefschetz elements. 
\end{itemize}
\end{Theorem}

This paper is organized as follows: 
In Section \ref{SLP}, we recall the strong Lefschetz property and Hodge--Riemann relation. 
Then we see the relation between the strong Lefschetz property and the Hessian matrices, and between the Hodge--Riemann relation and the Hessian matrices. 
In Section \ref{Main results}, we discuss the strong Lefschetz property and Hodge--Riemann relation for the Artinian Gorenstein algebras defined by the Platonic solids. 
In Section \ref{Hessian matrices}, we calculate the Hessian matrices of the homogeneous polynomials of the Platonic solids. 

\begin{acknowledgment}
This work was supported by the Sasakawa Scientific Research Grant from The Japan Science Society.
\end{acknowledgment}


\section{Strong Lefchetz property, Hodge--Riemann relation and Hessian matrices}\label{SLP}

We recall some properties for a graded algebra, the strong Lefschetz property and Hodge--Riemann relation. 

Let $A=\bigoplus_{k=0}^{s} A_{k}$, $A_{s}\neq \zero$, be a graded Artinian algebra with a symmetric bilinear map $P^{k}$ from $A_{k}\times A_{s-k}$ to $\RR$. 
For a graded algebra $A=\bigoplus_{k=0}^{s} A_{k}$, 
define $h_{k}$ by the dimension of $k$-th homogeneous component of $A$. 
We call the sequence $(h_{0}, h_{1}, \ldots, h_{s})$ the \emph{Hilbert series} of $A$. 

We say that $A$ has the \emph{strong Lefschetz property} 
if there exists an element $\ell \in A_{1}$ such that the linear map 
\begin{align}\label{slp}
\times \ell^{s-2k}\colon A_{k}\to A_{s-k}, && f\mapsto \ell^{s-2k}f
\end{align}
is bijective for each nonnegative integer $k\leq \frac{s}{2}$. 
We call $\ell \in A_{1}$ with this property a \emph{strong Lefschetz element}. 
If $A$ has the strong Lefschetz property, then the Hilbert series of $A$ is palindromic. 

We say that $A$ satisfies the \emph{Hodge--Riemann relation} with respect to $\ell\in A_{1}$
if the symmetric bilinear form
\begin{align}\label{hrr}
Q^{k}_{\ell}\colon A_{k}\times A_{k}\to\RR, && (f,g)\mapsto (-1)^{k}P^{k}(f, \ell^{s-2k}g)
\end{align}
is positive definite on the kernel $\times \ell^{s-2k+1}\colon A_{k}\to A_{s-k+1}$ for each nonnegative integer $k\leq \frac{s}{2}$. 

The strong Lefschetz property and Hodge--Riemann relation are defined on a general graded algebra with a symmetric bilinear form, but we consider those properties on a graded Artinian Gorenstin algebra with the Poincar\'e duality in this paper. 

Let $\der_{i}=\frac{\der}{\der x_{i}}$ be the partial derivative operator of $x_{i}$. 
The polynomial ring $\RR[\der_{1}, \der_{2}, \ldots, \der_{n}]$ acts on $\RR[x_{1}, x_{2}, \ldots, x_{n}]$ in the usual manner. 
For a homogeneous polynomial $F\in \RR[x_{1}, x_{2}, \ldots, x_{n}]$, we define the annihilator $\Ann(F)$ by
\begin{align*}
\Ann(F)=\Set{f\in \RR[\der_{1}, \ldots, \der_{n}]| fF=0}. 
\end{align*}
Then $\Ann(F)$ is a homogeneous ideal of $\RR[\der_{1}, \ldots, \der_{n}]$. 
Let $A=\RR[\der_{1}, \ldots, \der_{n}]/\Ann(F)$. 
Since $\Ann(F)$ is homogeneous, the algebra $A$ is graded. 
Furthermore $A$ is an Artinian Gorenstein algebra. 
Conversely, a graded Artinian Gorenstein algebra $A$ has the presentation
\begin{align*}
A=\RR[\der_{1}, \ldots, \der_{n}]/\Ann(F)
\end{align*}
for some homogeneous polynomial $F\in \RR[x_{1}, x_{2}, \ldots, x_{n}]$. 
The socle degree $s$ of $A$ is the degree of $F$. 
Thus the maps
\begin{align*}
P_{F}^{k}: A_{k}\times A_{s-k}\to \RR, && (f,g)\mapsto fgF
\end{align*}
are bilinear maps. 
We call $P_{F}=\bigoplus_{k}P_{F}^{k}$ the \emph{Poincar\'e duality} of $A$. 
The Hilbert series is palindromic for every graded Artinian Gorenstein algebra. 

Let $A=\RR[\der_{1}, \ldots, \der_{n}]/\Ann(F)=\bigoplus_{k=0}^{s}A_{k}$ be a graded Artinian Gorenstein algebra, 
and $\Lambda_{k}$ the basis for $A_{k}$. 
We define the matrix $H_{F}^{k}$ by 
\begin{align*}
H_{F}^{k}=
\left(
e_{i}
e_{j}
F
\right)_{e_{i},e_{j}\in \Lambda_{k}}. 
\end{align*}
The matrix $H_{F}^{k}$ is called the \emph{$k$-th Hessian matrix} and $\det H_{F}^{k}$ is called the $k$th \emph{Hessian} of $F$ with respect to the basis $\Lambda_{k}$. 
We define the $0$th Hessian of $F$ to be $F$. 
If $\Lambda_{1}=\Set{\der_{1}, \der_{2}, \ldots, \der_{n}}$, then $H_{F}^{1}$ coincides with the usual Hessian matrix of $F$.

\begin{Theorem}[Watanabe \cite{W2000}, Maeno--Watanabe \cite{MR2594646}]\label{SLP and Hess}
Let $\bm{a}=(a_{1}, a_{2}, \ldots, a_{n})\in\RR^{n}$, and $\ell_{\bm{a}}=a_{1}\der_{1}+a_{2}\der_{2}+\cdots+a_{n}\der_{n}$. 
The multiplication map 
$\times \ell_{\bm{a}}^{s-2k}\colon A_{k}\to A_{s-k}$ is bijective 
if and only if $\det H_{F}^{k}(\bm{a})\neq 0$. 
\end{Theorem}

\begin{Remark}
By Theorem \ref{SLP and Hess}, a strong Lefschetz element comes from an open dense space where the determinants do not vanish. 
Thus, if the $k$-th Hessian does not vanish as a polynomial for all $k$, then the Artinian Gorenstein algebra $A$ has the strong Lefschetz property.  
\end{Remark}


\begin{Proposition}\label{HRR and Hess}
Assume that $F(\bm{a})>0$ for $\bm{a}=(a_{1}, a_{2}, \ldots, a_{n})\in \RR^{n}$. 
Let $\ell_{\bm{a}}=a_{1}\der_{1}+a_{2}\der_{2}+\cdots+a_{n}\der_{n}$. 
The algebra $A$ satisfies the condition (\ref{hrr}) when $k=1$ with respect to $\ell_{\bm{a}}$
if and only if 
the first Hessian matrix $H_{F}^{1}(\bm{a})$ has $n-1$ negative eigenvalues and one positive eigenvalue. 
\end{Proposition}

\begin{proof}
Since the map 
\begin{align*}
\times \ell_{\bm{a}}^{s}:A_0 \xrightarrow{\times \ell_{\bm{a}}} A_1 \xrightarrow{\times \ell_{\bm{a}}^{s-1}} A_s
\end{align*}
is bijective, we have $A_1=\RR \ell_{\bm{a}}\oplus \Ker(\times \ell_{\bm{a}}^{s-1})$. 
And we have 
\begin{align*}
Q_{\ell_{\bm{a}}}^{1}(\ell_{\bm{a}}, \ell_{\bm{a}})&=-P_{F}^{1}(\ell_{\bm{a}}, \ell_{\bm{a}}^{s-2}\ell_{\bm{a}}) \\
&=-P_{F}^{1}(\ell_{\bm{a}}, \ell_{\bm{a}}^{s-1}) \\
&=-\ell_{\bm{a}}^{s}F \\
&=-s!F(\bm{a}) <0. 
\end{align*}
Hence, 
the algebra $A$ satisfies the condition (\ref{hrr}) when $k=1$ with respect to $\ell_{\bm{a}}$ 
if and only if 
$Q_{\ell_{\bm{a}}}^{1}$ has $n-1$ positive eigenvalues and one negative eigenvalue. 

Furthermore, the representing matrix associated to $Q_{\ell_{\bm{a}}}^{1}$ with respect to a basis $\Lambda_{1}$ for $A_1$ is given by the first Hessian matrix $-H_{F}^{1}(\bm{a})$. 
In fact, for $e_{i}, e_{j}\in \Lambda_{1}$, 
\begin{align*}
Q_{\ell_{\bm{a}}}^{1}(e_{i}, e_{j})&=-P_{F}^{1}(e_{i}, \ell_{\bm{a}}^{s-2}e_{j}) \\
&=-\ell_{\bm{a}}^{s-2}e_{i}e_{j}F \\
&=-(a_{1}\der_{1}+a_{2}\der_{2}+\cdots+a_{n}\der_{n})^{s-2}e_{i}e_{j}F \\
&=-(s-2)!(e_{i}e_{j}F)(\bm{a}) \\
&=-(s-2)!\left(H_{F}^{1}(\bm{a})\right)_{i,j}. 
\end{align*}
\end{proof}

\section{Main results}\label{Main results}

In this section, we discuss the strong Lefschetz property and Hodge--Riemann relation for the Artinian Gorenstein algebras defined by the regular polyhedra.

For a polyhedron $\PPP$, we call a $0$-dimensional face a \emph{vertex} of $\PPP$, an $1$-dimensional face an \emph{edge} of $\PPP$, and a $2$-dimensional face, simply, a \emph{face} of $\PPP$. 
Let $V(\PPP)$, $E(\PPP)$ and $F(\PPP)$ denote the collection of vertices, edges and faces of $\PPP$, respectively. 
We focus a combinatorial data, the face poset of $\PPP$. 
The face poset of $\PPP$ is the set $V(\PPP)\cup E(\PPP)\cup F(\PPP)$ ordered by inclusion. 
We regard edges and faces of $\PPP$ as subsets of $V(\PPP)$.   

In this paper, we are only interested in the regular polyhedra. 
For a regular polyhedra $\PPP$, we define a homogeneous polynomial by
\begin{align*}
F_{\PPP}=\sum_{F\in F(\PPP)}\prod_{v\in F}x_{v}. 
\end{align*}
The number of vertices appears as the number of variables, and the number of faces appears as the numbers of terms.  
The shape of faces appears as the degree. 
For example, if faces of $\PPP$ are $d$-gons, then $F_{\PPP}$ is a homogeneous of degree $d$. 

For a polyhedron $\PPP$, $A_{\PPP}=\bigoplus_{k=0}^{s}A_{k}$ denotes the Artinian Gorenstein algebra defined by the homogeneous polynomial $F_{\PPP}$. 
We call $A_{\PPP}$ the algebra associated to $\PPP$. 
For the algebra $A_{\PPP}$, we consider the Poincar\'e duality $P_{F_{\PPP}}=\bigoplus_{k}P_{F_{\PPP}}^{k}$, where 
\begin{align*}
P_{F_{\PPP}}^{k}: A_{k}\times A_{s-k}\to \RR, && (f,g)\mapsto fgF_{\PPP}. 
\end{align*}
We see the strong Lefschetz property of $A_{\PPP}$ and Hodge--Riemann relation with respect to the Poincar\'e duality  $P_{F_{\PPP}}=\bigoplus_{k}P^{k}_{F_{\PPP}}$ of $A_{\PPP}$. 

\begin{Remark}
As mentioned in Section \ref{SLP}, if $\Set{\der_{i}}_{i\in V(\PPP)}$ is a basis for $A_{1}$, then the first Hessian matrix $H^{1}_{F_{\PPP}}$ coincides with the usual Hessian matrix of $F_{\PPP}$. 
Unless noted, we calculate $H^{1}_{F_{\PPP}}$ as the Hessian matrix of $F_{\PPP}$. 
In other words, we calculate $H^{1}_{F_{\PPP}}$ with respect to $\Set{\der_{i}}_{i\in V(\PPP)}\subset A_{1}$.   
\end{Remark}

\subsection{Regular tetrahedron}\label{n=4}

Let us consider the regular tetrahedron $\PPP$. 
The regular tetrahedron has $4$ vertices, $6$ edges and $4$ faces. 
We assign the number $1,2,3,4$ to the vertices. 
Let $$F(\PPP)=\Set{\Set{1,2,3}, \Set{1,2,4}, \Set{1,3,4}, \Set{2,3,4}}. $$
See Figure \ref{4}. 
\begin{figure}
\centering
\begin{minipage}{0.3\textwidth}
\includegraphics[width=0.98\textwidth]{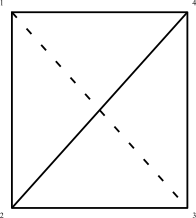}
\end{minipage}
\caption{Tetrahedron}\label{4}
\end{figure}
Then 
\begin{align*}
F_{\PPP}=x_{1}x_{2}x_{3}+x_{1}x_{2}x_{4}+x_{1}x_{3}x_{4}+x_{2}x_{3}x_{4}. 
\end{align*}
The Hilbert series of the algebra associated to $\PPP$ is $(1,4,4,1)$. 

The homogeneous polynomial $F_{\PPP}$ is equal to the elementary symmetric polynomial $e_{3}(x_{1}, x_{2}, x_{3}, x_{4})$ of degree $3$ in $4$ variables. 
In \cite{MR3566530}, Maeno and Numata show that the algebra defined by the annihilator of the elementary symmetric polynomial $e_{d}(x_{1}, x_{2}, \ldots, x_{n})$ of degree $d$ in $n$ variables satisfies the strong Lefschetz property for every $d$ and $n$. 

Furthermore, $F_{\PPP}$ is equal to the Kirchhoff polynomial of the cycle graph $C_{4}$ with $4$ vertices.  
In \cite{NY2019}, Nagaoka and Yazawa show that the algebra defined by a Kirchhoff polynomial, a homogeneous polynomial defined by a graph, satisfies the strong Lefschetz property and Hodge--Riemann relation ``at degree one'' (it will be explained in Remark \ref{rem deg1}) on the positive orthant. 
More generally, in \cite{MNY2020}, Murai, Nagaoka and Yazawa show that the algebra defined by the basis generating polynomial, a generalization of a Kirchhoff polynomial, satisfies the strong Lefschetz property and Hodge--Riemann relation at degree one on the positive orthant. 

\begin{Remark}\label{rem deg1}
Let $A=\bigoplus_{k=0}^{s} A_{k}$ be a graded algebra with a symmetric bilinear map $P^{k}$ from $A_{k}\times A_{s-k}$ to $\RR$. 
We say that $A$ satisfies the strong Lefschetz property at degree one if the condition (\ref{slp}) holds when $k=1$, and that $A$ satisfies the Hodge--Riemann relation with respect to $\ell\in A_1$ if the condition (\ref{hrr}) holds when $k=1$. 
\end{Remark}

To summarize of this section, we obtain the following. 

\begin{Theorem}\label{n=4}
The algebra associated to the regular tetrahedron satisfies the strong Lefschetz property and Hodge--Riemann relation on the positive orthant. 
\end{Theorem}

\subsection{Regular hexahedron}\label{n=6}

Let us consider the regular hexahedron $\PPP$. 
The regular hexahedron has $8$ vertices, $12$ edges and $6$ faces. 
We assign the number $1,2,\ldots, 8$ to the vertices. 
Let $$F(\PPP)=\Set{\Set{1,2,3,4}, \Set{2,3,6,7}, \Set{3,4,7,8}, \Set{1,4,5,8}, \Set{1,2,5,6}, \Set{5,6,7,8}}. $$ 
See Figure \ref{6}. 
\begin{figure}
\centering
\begin{minipage}{0.3\textwidth}
\includegraphics[width=0.98\textwidth]{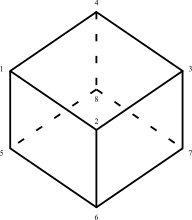}
\end{minipage}
\caption{Hexahedron}\label{6}
\end{figure}
Then 
\begin{align*}
F_{\PPP}=x_{1}x_{2}x_{3}x_{4}+x_{2}x_{3}x_{6}x_{7}+x_{3}x_{4}x_{7}x_{8}+x_{1}x_{4}x_{5}x_{8}+x_{1}x_{2}x_{5}x_{6}+x_{5}x_{6}x_{7}x_{8}. 
\end{align*}

\begin{Proposition}\label{eigen n=6}
The eigenvalues of the matrix $H_{F_\PPP}^{1}(1,1,\ldots, 1)$ are $9, 1, -3$ with the dimensions of the eigenspaces $1,3,4$, respectively.  
\end{Proposition}
The matrix $H_{F_\PPP}^{1}(1,1,\ldots, 1)$ is in Section \ref{Hessian matrices}. 
Also the proof of Proposition \ref{eigen n=6} is in Section \ref{Hessian matrices}. 

From Proposition \ref{eigen n=6}, Theorem \ref{SLP and Hess} and Proposition \ref{HRR and Hess}, we obtain the following. 

\begin{Theorem}\label{n=6}
The algebra associated to the regular hexahedron satisfies the strong Lefschetz property with a strong Lefschetz element $\sum_{i=1}^{8}\der_{i}$, 
but does not satisfy the Hodge--Riemann relation with respect to $\sum_{i=1}^{8}\der_{i}$. 
\end{Theorem}

\begin{Remark}
If $\Set{v, v'}$ does not a subset of a face of $\PPP$, then $\der_{v}\der_{v'}$ is in $\Ann(F_{\PPP})$. 
In this case, 
\begin{align*}
\der_{1}\der_{7}, \der_{2}\der_{8}, \der_{3}\der_{5}, \der_{4}\der_{6}\in \Ann(F_{\PPP}). 
\end{align*}
Further, monomials of degree $2$ of $A_{\PPP}$ have the following relations: 
\begin{align}\label{eqs}
\begin{array}{ccc}
\der_{1}\der_{2}=\der_{7}\der_{8}, & \der_{1}\der_{4}=\der_{6}\der_{7}, & \der_{1}\der_{5}=\der_{3}\der_{7}, \\
\der_{2}\der_{3}=\der_{5}\der_{8}, & \der_{2}\der_{6}=\der_{4}\der_{8}, & \der_{3}\der_{4}=\der_{5}\der_{6}. 
\end{array}
\end{align}
The other monic monomials of degree two and monomials in left hand sides of the equations \eqref{eqs} generate $A_{2}$. 
Furthermore, since the second Hessian with respect to them does not vanish, they form basis for $A_{2}$.  
Thus the Hilbert series of the algebra associated to $\PPP$ is $(1,8,18,8,1)$. 
The second Hessian matrix $H_{F_\PPP}^{2}$ is in Section \ref{Hessian matrices}. 
\end{Remark}

\subsection{Regular octahedron}\label{n=8}

Let us consider the regular octahedron $\PPP$. 
The regular octahedron has $6$ vertices, $12$ edges and $8$ faces. 
We assign the number $1,2,\ldots, 6$ to the vertices. 
Let $$F(\PPP)=\Set{\Set{1,2,3}, \Set{1,3,4}, \Set{1,4,5}, \Set{1,2,5}, \Set{2,3,6}, \Set{3,4,6}, \Set{4,5,6}, \Set{2,5,6}}. $$ 
See Figure \ref{8}. 
\begin{figure}
\centering
\begin{minipage}{0.3\textwidth}
\includegraphics[width=0.98\textwidth]{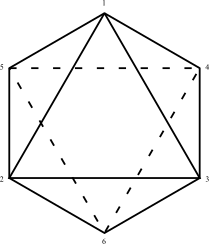}
\end{minipage}
\caption{Octahedron}\label{8}
\end{figure}
Then 
\begin{align*}
F_{\PPP}=x_{1}x_{2}x_{3}+x_{1}x_{3}x_{4}+x_{1}x_{4}x_{5}+x_{1}x_{2}x_{5}+x_{2}x_{3}x_{6}+x_{3}x_{4}x_{6}+x_{4}x_{5}x_{6}+x_{2}x_{5}x_{6}. 
\end{align*}
The algebra $A_{\PPP}$ has $6$ variables. 

We can reduce the number of variables. 
\begin{Proposition}\label{prop reduce}
For the regular octahedron $\PPP$, we have 
\begin{align*}
A_{\PPP}\cong \RR[t_{1}, t_{2}, t_{3}]/(t_{1}^{2}, t_{2}^{2}, t_{3}^{2}). 
\end{align*}
\end{Proposition}

\begin{proof}
Let 
\begin{align*}
\Phi: \RR[\der_{1}, \der_{2}, \ldots, \der_{6}] &\to \RR[t_{1}, t_{2}, t_{3}]/(t_{1}^{2}, t_{2}^{2}, t_{3}^{2}) \\
\der_{1}, \der_{6} &\mapsto t_{1}, \\
\der_{2}, \der_{4} &\mapsto t_{2}, \\
\der_{3}, \der_{5} &\mapsto t_{3}.  
\end{align*}
We calculate $\Ker(\Phi)$. 
In our situation, we have a way to calculate the kernel: 
For sets $X=\Set{x_{1}, x_{2}, \ldots, x_{n}}$, $Y=\Set{y_{1}, y_{2}, \ldots, y_{m}}$ of variables, and 
a subset $D$ of the polynomial ring $\RR[Y]$, consider
\begin{align*}
\varphi&: \RR[X] \to \RR[Y]/\langle D\rangle, \\
G&=\Set{x_{i}-\varphi(x_{i})|i\in\Set{1,2,\ldots, n}}\cup D.  
\end{align*}
Then we have $\Ker(\varphi)=\langle G \rangle_{\RR[X, Y]}\cap \RR[X]$. 

For $\Phi$, the map defined above, we consider $D=\Set{t_{1}^{2}, t_{2}^{2}, t_{3}^{2}}$. 
Then we have 
\begin{align*}
G=\Set{\der_{1}-t_{1}, \der_{6}-t_{1}, \der_{2}-t_{2}, \der_{4}-t_{2}, \der_{3}-t_{3}, \der_{5}-t_{3}, t_{1}^{2}, t_{2}^{2}, t_{3}^{2}}. 
\end{align*}
A Gr\"{o}bner basis for $\langle G \rangle_{\RR[X, Y]}$ with respect to the lexicographical order with 
$t_{3}>t_{2}>t_{1}>\der_{6}>\der_{5}>\der_{4}>\der_{3}>\der_{2}>\der_{1}$ is 
\begin{align*}
\Set{
\der_{1}-t_{1}, \der_{2}-t_{2}, \der_{3}-t_{3}, 
\der_{1}-\der_{6}, \der_{2}-\der_{4}, \der_{3}-\der_{5}, 
\der_{1}^{2}, \der_{2}^{2}, \der_{3}^{2}
}. 
\end{align*}
Hence, 
$\Ker(\varphi)=\langle \der_{1}-\der_{6}, \der_{2}-\der_{4}, \der_{3}-\der_{5}, \der_{1}^{2}, \der_{2}^{2}, \der_{3}^{2} \rangle_{\RR[X]}$. 

The Hilbert series of $A_{\PPP}=\RR[\der_{1}, \der_{1}, \ldots, \der_{6}]/\Ann(F_{\PPP})$ is $(1,3,3,1)$ 
since the degree of $F_{\PPP}$ is $3$, and the first Hessian $H^{1}_{F_{\PPP}}$ with respect to the $\Set{\der_{1}, \der_{2}, \der_{3}}$ does not vanish. 

It is obvious that $\Ker(\varphi)\subset \Ann(F)$, and the Hilbert series of $\RR[\der_{1}, \der_{1}, \ldots, \der_{6}]/\Ker(\varphi)$ is $(1,3,3,1)$, so we obtain 
\begin{align*}
A_{\PPP}\cong \RR[\der_{1}, \der_{1}, \ldots, \der_{6}]/\Ker(\varphi) \cong \RR[t_{1}, t_{2}, t_{3}]/(t_{1}^{2}, t_{2}^{2}, t_{3}^{2}). 
\end{align*}
\end{proof}

In general, the algebra $\RR[\der_{1}, \der_{2}, \ldots, \der_{n}]/(\der_{1}^{a_{1}+1}, \der_{2}^{a_{2}+1}, \ldots,  \der_{n}^{a_{n}+1})$, where $a_{i}\geq 0$ for all $i$ satisfies the strong Lefschetz property with a strong Lefschetz element $\sum_{i=1}^{n}\der_{i}$ (\cite{MR3112920}). 
If $F=x_{1}^{a_{1}}x_{2}^{a_{2}}\cdots x_{n}^{a_{n}}$, then we have $A_{F}=\RR[\der_{1}, \der_{2}, \ldots, \der_{n}]/(\der_{1}^{a_{1}+1}, \der_{2}^{a_{2}+1}, \ldots,  \der_{n}^{a_{n}+1})$. 
In \cite{NY2019} and \cite{MNY2020}, when $a_{1}=a_{2}=\cdots=a_{n}=1$, namely $F=x_{1}x_{2}\cdots x_{n}$, 
they show that the algebra $\RR[\der_{1}, \der_{2}, \ldots, \der_{n}]/(\der_{1}^{2}, \der_{2}^{2}, \ldots,  \der_{n}^{2})$ satisfies the strong Lefschetz property and Hodge--Riemann relation at degree one on the positive orthant. 
In fact, the homogeneous polynomial $F=x_{1}x_{2}\cdots x_{n}$ coincides with the Kirchhoff polynomial a tree graph with $n+1$ vertices. 

To summarize of this section, we obtain the following. 

\begin{Theorem}\label{n=8}
The algebra associated to the regular octahedron satisfies the strong Lefschetz property and Hodge--Riemann relation on the positive orthant. 
\end{Theorem}

\subsection{Regular dodecahedron}\label{n=12}

Let us consider the regular dodecahedron $\PPP$. 
The regular dodecahedron has $20$ vertices, $30$ edges and $12$ faces. 
We assign the number $1,2,\ldots, 20$ to the vertices. 
Let 
\begin{align*}
F(\PPP)=\left\{
\begin{array}{cccc}
\Set{1,2,3,4,5}, 
\Set{1,2,6,7,19}, 
\Set{2,3,7,8,20}, 
\Set{3,4,8,9,16}, \\
\Set{4,5,9,10,17}, 
\Set{1,5,6,10,18}, 
\Set{11,12,13,14,15}, \\
\Set{11,12,16,17,9}, 
\Set{12,13,17,18,10},  
\Set{13,14,18,19,6}, \\
\Set{14,15,19,20,7}, 
\Set{11,15,16,20,8}
\end{array}
\right\}. 
\end{align*}
See Figure \ref{12}. 
\begin{figure}
\centering
\begin{minipage}{0.3\textwidth}
\includegraphics[width=0.98\textwidth]{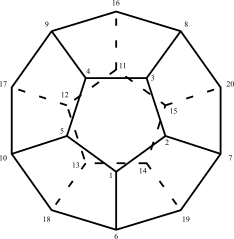}
\end{minipage}
\caption{Dodecahedron}\label{12}
\end{figure}
Then 
\begin{align*}
F_{\PPP}=
&x_{1}x_{2}x_{3}x_{4}x_{5}+x_{1}x_{2}x_{6}x_{7}x_{19}+x_{2}x_{3}x_{7}x_{8}x_{20}+x_{3}x_{4}x_{8}x_{9}x_{16}\\
+&x_{4}x_{5}x_{9}x_{10}x_{17}+x_{1}x_{5}x_{6}x_{10}x_{18}+x_{11}x_{12}x_{13}x_{14}x_{15}+x_{11}x_{12}x_{16}x_{17}x_{9}\\
+&x_{12}x_{13}x_{17}x_{18}x_{10}+x_{13}x_{14}x_{18}x_{19}x_{6}+x_{14}x_{15}x_{19}x_{20}x_{7}+x_{11}x_{15}x_{16}x_{20}x_{8}. 
\end{align*}

The socle degree of $A_{\PPP}$ is $5$, in other words, $A_{\PPP}=\bigoplus_{k=0}^{5}A_{k}$. 
We calculate the first and second Hessian matrices of $F_{\PPP}$. 

\begin{Proposition}\label{eigen n=12 Hess1}
Let $\bm{a}=(a_{i})_{i}\in\RR^{20}$ with $a_{i}=1$ for $i\not\in\Set{6,7,8,9,10}$ and $a_{i}=0$ for $i\in\Set{6,7,8,9,10}$. 
The matrix $H_{F_\PPP}^{1}(\bm{a})$ is non-degenerate. 
Moreover, the number of the positive eigenvalues is more than two.     
\end{Proposition}

\begin{Proposition}\label{eigen n=12 Hess2}
Let $\bm{a}=(a_{i})_{i}\in\RR^{20}$ with $a_{1}=0$ and $a_{i}=1$ for $i\neq 1$. 
The matrix $H_{F_\PPP}^{2}(\bm{a})$ with some basis for $A_2$ is not degenerate. 
\end{Proposition}

The matrices  $H_{F_\PPP}^{1}(\bm{a})$ and $H_{F_\PPP}^{2}(\bm{a})$ are in Section \ref{Hessian matrices}. 
Also the proof of Propositions \ref{eigen n=12 Hess1} and \ref{eigen n=12 Hess2} are in Section \ref{Hessian matrices}. 
To show Proposition \ref{eigen n=12 Hess2}, we use CoCalc \cite{CoCalc}. 

To summarize of this section, we obtain the following. 

\begin{Theorem}\label{n=12}
The algebra associated to the regular dodecahedron satisfies the strong Lefschetz property, 
but does not satisfy the Hodge--Riemann relation with respect to $\sum_{i\neq 6,7,8,9,10}\der_{i}$. 
\end{Theorem}

\begin{Remark}
Let 
$\bm{a}=(a_{i})_{i}\in\RR^{20}$ with $a_{i}=1$ for all $i$; 
$\bm{b}=(b_{i})_{i}\in\RR^{20}$ with $b_{i}=1$ for $i\not\in\Set{6,7,8,9,10}$ and $b_{i}=0$ for $i\in\Set{6,7,8,9,10}$; 
and 
$\bm{c}=(c_{i})_{i}\in\RR^{20}$ with $c_{1}=0$ and $c_{i}=1$ for $i\neq 1$. 
We have 
\begin{itemize}
\item $\det H^{1}_{F_{\PPP}}(\bm{a})=0$ and $\det H^{2}_{F_{\PPP}}(\bm{a})=0$, 
\item $\det H^{1}_{F_{\PPP}}(\bm{b})\neq 0$ and $\det H^{2}_{F_{\PPP}}(\bm{b})=0$, 
\item $\det H^{1}_{F_{\PPP}}(\bm{c})\neq 0$ and $\det H^{2}_{F_{\PPP}}(\bm{c})\neq 0$. 
\end{itemize}
Hence, $\ell_{\bm{a}}=\sum_{i=1}^{20}a_{i}\der_{i}$ and $\ell_{\bm{b}}=\sum_{i=1}^{20}b_{i}\der_{i}$ are not strong Lefschetz elements. 
The linear element $\ell_{\bm{c}}=\sum_{i=1}^{20}c_{i}\der_{i}$ is a strong Lefschetz element. 
We use Cocalc for calculation the determinant $\det H^{1}_{F_{\PPP}}(\bm{c})$. 
\end{Remark}

\begin{Remark}
The Hilbert series of $A_{\PPP}$ is $(1,20,90,20,1)$. 
\end{Remark}

\subsection{Regular icosahedron}\label{n=20}

Let us consider the regular icosahedron $\PPP$. 
The regular icosahedron has $12$ vertices, $30$ edges and $20$ faces. 
We assign the number $1,2,\ldots, 12$ to the vertices. 
Let
\begin{align*}
F(\PPP)=
\left\{
\begin{array}{cccc}
\Set{1,2,3}, \Set{1,3,4}, \Set{1,4,5}, \Set{3,4,9}, \Set{1,5,6}, \\
\Set{1,2,6}, \Set{2,3,8}, \Set{3,8,9}, \Set{2,6,7}, \Set{2,7,8}, \\
\Set{10,11,12}, \Set{4,5,10}, \Set{4,9,10}, \Set{5,6,11}, \Set{5,10,11}, \\
\Set{9,10,12}, \Set{8,9,12}, \Set{6,7,11}, \Set{7,11,12}, \Set{7,8,12}
\end{array}
\right\}. 
\end{align*}
See Figure \ref{20}. 
\begin{figure}
\centering
\begin{minipage}{0.3\textwidth}
\includegraphics[width=0.98\textwidth]{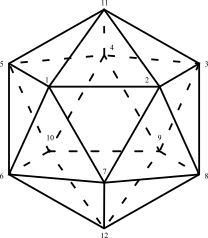}
\end{minipage}
\caption{Icosahedron}\label{20}
\end{figure}
Then 
\begin{align*}
F_{\PPP}=
&x_{1}x_{2}x_{3}+x_{1}x_{3}x_{4}+x_{1}x_{4}x_{5}+x_{3}x_{4}x_{9}+x_{1}x_{5}x_{6}\\
+&x_{1}x_{2}x_{6}+x_{2}x_{3}x_{8}+x_{3}x_{8}x_{9}+x_{2}x_{6}x_{7}+x_{2}x_{7}x_{8}\\
+&x_{10}x_{11}x_{12}+x_{4}x_{5}x_{10}+x_{4}x_{9}x_{10}+x_{5}x_{6}x_{11}+x_{5}x_{10}x_{11}\\
+&x_{9}x_{10}x_{12}+x_{8}x_{9}x_{12}+x_{6}x_{7}x_{11}+x_{7}x_{11}x_{12}+x_{7}x_{8}x_{12}. 
\end{align*}

\begin{Proposition}\label{eigen n=20}
The eigenvalues of the matrix $H_{F_\PPP}^{1}(1,1,\ldots, 1)$ of the homogeneous polynomial $F_{\PPP}$ are $10,-2,2\sqrt{5},-2\sqrt{5}$ with the dimensions of the eigenspaces $1,5,3,3$, respectively. 

\end{Proposition}
The matrix $H_{F_\PPP}^{1}(1,1,\ldots, 1)$ is in Section \ref{Hessian matrices}. 
Also the proof of Proposition \ref{eigen n=20} is in Section \ref{Hessian matrices}. 

To summarize of this section, we obtain the following. 

\begin{Theorem}\label{n=20}
The algebra associated to the regular icosahedron satisfies the strong Lefschetz property with a strong Lefschetz element $\sum_{i=1}^{20}\der_{i}$, 
but does not satisfy the Hodge--Riemann relation with respect to $\sum_{i=1}^{20}\der_{i}$. 
\end{Theorem}

\section{Hessian matrices}\label{Hessian matrices}

Here, we calculate the Hessian matrices of the homogeneous polynomials of the regular polyhedra. 

\subsection{Regular tetrahedron}\label{appen n=4}

Let $\PPP$ be the regular tetrahedron. 
The matrix $H_{F_\PPP}^{1}(1,1,1,1)$ is 
\begin{align*}
\begin{pmatrix}
0&2&2&2 \\
2&0&2&2 \\
2&2&0&2 \\
2&2&2&0 \\
\end{pmatrix}. 
\end{align*}
The eigenvalues of the matrix are $6, -2, -2, -2$. 

\subsection{Regular hexahedron}\label{appen n=6}

Let $\PPP$ be the regular hexahedron. 
The matrix $H_{F_\PPP}^{1}(1,1,\ldots, 1)$ is the following, and has the following block decomposition:  
\begin{align*}
\left(
\begin{array}{cccc|cccc}
0&2&1&2&2&1&0&1 \\
2&0&2&1&1&2&1&0 \\
1&2&0&2&0&1&2&1 \\
2&1&2&0&1&0&1&2 \\ \hline
2&1&0&1&0&2&1&2 \\
1&2&1&0&2&0&2&1 \\
0&1&2&1&1&2&0&2 \\
1&0&1&2&2&1&2&0 \\
\end{array}
\right). 
\end{align*}
By \cite[Section 2]{Y2018}, the eigenvalues of the matrix come from the eigenvalues of the following four matrices $M_0, M_1, M_2, M_3$: 
\begin{align*}
M_0=
\begin{pmatrix}
5&4 \\
4&5
\end{pmatrix}, &&
M_1=M_3=
\begin{pmatrix}
-1&2 \\
2&-1
\end{pmatrix}, &&
M_2=
\begin{pmatrix}
-3&0 \\
0&-3
\end{pmatrix}. 
\end{align*}
The eigenvalues of $M_0$ are $9,1$. 
The eigenvalues of $M_1=M_3$ are $1,-3$. 
The eigenvalues of $M_2$ are $-3,-3$. 
Hence we have Proposition \ref{eigen n=6}. 

By Section \ref{Main results}, we know that 
\begin{align*}
\Lambda_{2}=
\Set{
\der_{1}\der_{3}, \der_{2}\der_{4}, \der_{1}\der_{8}, \der_{4}\der_{5}, \der_{2}\der_{5}, 
\der_{1}\der_{6}, \der_{3}\der_{6}, \der_{2}\der_{7}, \der_{4}\der_{7}, \der_{3}\der_{8}, 
\der_{6}\der_{8}, \der_{5}\der_{7}, \der_{1}\der_{2}, \der_{1}\der_{4}, \der_{1}\der_{5}, 
\der_{2}\der_{6}, \der_{2}\der_{3}, \der_{3}\der_{4}
}
\end{align*}
generates $A_{2}$. 
Let us calculate the second Hessian matrix with respect to $\Lambda_{2}$ is  
\begin{align*}
\begin{pmatrix}
A&\zero&\zero&\zero&\zero&\zero&\zero&\zero&\zero \\
\zero&A&\zero&\zero&\zero&\zero&\zero&\zero&\zero \\
\zero&\zero&A&\zero&\zero&\zero&\zero&\zero&\zero \\
\zero&\zero&\zero&A&\zero&\zero&\zero&\zero&\zero \\
\zero&\zero&\zero&\zero&A&\zero&\zero&\zero&\zero \\
\zero&\zero&\zero&\zero&\zero&A&\zero&\zero&\zero \\
\zero&\zero&\zero&\zero&\zero&\zero&A&\zero&\zero \\
\zero&\zero&\zero&\zero&\zero&\zero&\zero&A&\zero \\
\zero&\zero&\zero&\zero&\zero&\zero&\zero&\zero&A \\
\end{pmatrix}, 
\end{align*}
where
\begin{align*}
A=
\begin{pmatrix}
0&1 \\
1&0
\end{pmatrix}, 
\zero=
\begin{pmatrix}
0&0 \\
0&0
\end{pmatrix}. 
\end{align*}
The eigenvalues are $1$ and $-1$. 
The multiplicity of both are $9$.  
Since the second Hessian with respect to $\Lambda_{2}$ does not vanish, $\Lambda_{2}$ is basis for $A_{2}$.  

\subsection{Regular octahedron}\label{appen n=8}

Let $\PPP$ be the regular octahedron. 
By Proposition \ref{prop reduce}, the algebra $A_{\PPP}$ is isomorphic to the graded Artinian Gorenstein algebra $\RR[\der_{1}, \der_{2}, \der_{3}]/(\der_{1}^{2}, \der_{2}^{2}, \der_{3}^{2})$. 
If $F=x_{1}x_{2}x_{3}$, then $\RR[\der_{1}, \der_{2}, \der_{3}]/(\der_{1}^{2}, \der_{2}^{2}, \der_{3}^{2})=\RR[\der_{1}, \der_{2}, \der_{3}]/\Ann(F)$. 

The first Hessian matrix of $F$ evaluated $1$ is the following. 
\begin{align*}
\begin{pmatrix}
0&1&1 \\
1&0&1 \\
1&1&0
\end{pmatrix}. 
\end{align*}
The eigenvalues of the matrix are $2, -1, -1$. 

\subsection{Regular dodecahedron}\label{appen n=12}

Let $\PPP$ be the regular dodecahedron. 
The socle degree of $A_{\PPP}$ is $5$, in other words, $A_{\PPP}=\bigoplus_{k=0}^{5}A_{k}$. 
We calculate the first and second Hessian matrices of $F_{\PPP}$. 

At first, we consider the first Hessian. 
The matrix $H_{F_\PPP}^{1}(1,1,\ldots, 1)$ is degenerate. 
In fact, if we set 
\begin{align*}
A=
\begin{pmatrix}
0&2&1&1&2 \\
2&0&2&1&1 \\
1&2&0&2&1 \\
1&1&2&0&2 \\
2&1&1&2&0 \\
\end{pmatrix}, &&
B=
\begin{pmatrix}
2&1&0&0&1 \\
1&2&1&0&0 \\
0&1&2&1&0 \\
0&0&1&2&1 \\
1&0&0&1&2 \\
\end{pmatrix}, &&
C=
\begin{pmatrix}
0&0&1&1&0 \\
0&0&0&1&1 \\
1&0&0&0&1 \\
1&1&0&0&0 \\
0&1&1&0&0 \\
\end{pmatrix}, \\
D=
\begin{pmatrix}
0&1&0&0&1 \\
1&0&1&0&0 \\
0&1&0&1&0 \\
0&0&1&0&1 \\
1&0&0&1&0 \\
\end{pmatrix}, &&
\zero=
\begin{pmatrix}
0&0&0&0&0 \\
0&0&0&0&0 \\
0&0&0&0&0 \\
0&0&0&0&0 \\
0&0&0&0&0 \\
\end{pmatrix}, &&
A'=
\begin{pmatrix}
0&1&1&1&1 \\
1&0&1&1&1 \\
1&1&0&1&1 \\
1&1&1&0&1 \\
1&1&1&1&0 \\
\end{pmatrix}, 
\end{align*}
then the matrix $H_{F_\PPP}^{1}(1,1,\ldots, 1)$ is
\begin{align*}
\begin{pmatrix}
A&B&\zero&C \\
B&D&C&2C \\
\zero&C&A&B \\
C&2C&B&D \\
\end{pmatrix}. 
\end{align*}
One can check that the matrix $H_{F_\PPP}^{1}(1,1,\ldots, 1)$ is degenerate. 
By Theorem \ref{SLP and Hess}, the linear element $\sum_{i=1}^{20}\der_{i}$ is not a strong Lefschetz element of $A_{\PPP}$.  
Instead of the matrix, we consider another matrix $H'$ evaluated the following:
\begin{align*}
\text{$x_{i}=1$ \ for $i \not\in \Set{6,7,8,9,10}$}, \\
\text{$x_{i}=0$ \ for $i \in \Set{6,7,8,9,10}$}. 
\end{align*}
The matrix $H'$ is the following, and has the following block decomposition 
\begin{align*}
\left(
\begin{array}{c|ccc}
A'&\zero&\zero&\zero \\ \hline
\zero&D&C&C \\
\zero&C&A'&\zero \\
\zero&C&\zero&\zero \\
\end{array}
\right). 
\end{align*}
The eigenvalues of $A'$ is $4$, $-1$. 
The multiplicities are $1$ and $3$, respectively. 
The matrix   
\begin{align*}
\left(
\begin{array}{ccc}
D&C&C \\
C&A'&\zero \\
C&\zero&\zero \\
\end{array}
\right)
\end{align*}
is non-degenerate since $A'$ and $C$ are non-degenerate. 
Moreover, since the trace is zero, the matrix has at least one positive eigenvalue. 
Therefore, $H'$ has at least two positive eigenvalues.  

Next, we consider the second Hessian. 
The second Hessian matrix of $F_{\PPP}$ is too large to calculate by hands, so we use CoCalc \cite{CoCalc} to calculation. 
\begin{lstlisting}[basicstyle=\ttfamily\footnotesize, frame=single, numbers=left]
xx=[var("x%d" % i) for i in range(20)]

Facet=[
[0, 1, 2, 3, 4], 
[0, 1, 5, 6, 18], 
[1, 2, 6, 7, 19], 
[2, 3, 7, 8, 15], 
[3, 4, 8, 9, 16], 
[0, 4, 5, 9, 17], 
[10, 11, 12, 13, 14], 
[8, 10, 11, 15, 16], 
[9, 11, 12, 16, 17], 
[5, 12, 13, 17, 18], 
[6, 13, 14, 18, 19], 
[7, 10, 14, 15, 19]
]

F=sum(prod(xx[i] for i in f) for f in Facet)

A2=Combinations(xx,2).list()
Basis=[]
for i in A2:
    if (F.derivative(i[0])).derivative(i[1])==0:
        pass
    else:
        Basis=Basis+[i]

ss={}
for xi in xx:
    ss[xi]=1
ss[x0]=0

h1=[F.derivative(i[0]).derivative(i[1]) for i in Basis]
h=[[hi.derivative(j[0]).derivative(j[1]) for j in Basis] for hi in h1]
H=Matrix([[ZZ(hij.substitute(ss)) for hij in hi] for hi in h])
det(H)
\end{lstlisting}

By this program, you obtain $\det H^{2}_{F_{\PPP}}(\bm{a})=342456532992$, where $\bm{a}=(a_{i})_{i}\in\RR^{20}$ with $a_{1}=0$ and $a_{i}=1$ for $i\neq 1$.

\subsection{Regular icosahedron}\label{appen n=20}

The first Hessian matrix $H_{F_\PPP}^{1}(1,1,\ldots, 1)$ of the homogeneous polynomial of the regular icosahedron is the following, and has the following block decomposition:

\begin{align*}
\left(
\begin{array}{ccccc|ccccc|c|c}
0&2&0&0&2&2&0&0&0&2&2&0 \\
2&0&2&0&0&2&2&0&0&0&2&0 \\
0&2&0&2&0&0&2&2&0&0&2&0 \\
0&0&2&0&2&0&0&2&2&0&2&0 \\
2&0&0&2&0&0&0&0&2&2&2&0 \\ \hline
2&2&0&0&0&0&2&0&0&2&0&2 \\
0&2&2&0&0&2&0&2&0&0&0&2 \\
0&0&2&2&0&0&2&0&2&0&0&2 \\
0&0&0&2&2&0&0&2&0&2&0&2 \\
2&0&0&0&2&2&0&0&2&0&0&2 \\ \hline
2&2&2&2&2&0&0&0&0&0&0&0 \\ \hline
0&0&0&0&0&2&2&2&2&2&0&0 \\
\end{array}
\right). 
\end{align*}

By \cite[Section 2]{Y2018}, the eigenvalues of the matrix are come from the eigenvalues of the following five matrices $M_0, M_1, M_2, M_3, M_4$: 
Let $\zeta_{5}$ be the primitive $5$th root of unity and $k\in\Set{1,2,3,4}$. 
Then we define the matrices $M_0, M_1, M_2, M_3, M_4$ by 
\begin{align*}
M_0=
\left(
\begin{array}{cc|cc}
4&4&2&0 \\
4&4&0&2 \\ \hline
10&0&0&0 \\
0&10&0&0 \\
\end{array}
\right),
&&
M_k=
\begin{pmatrix}
2(\zeta_{5}^{k}+\zeta_{5}^{-k})&2(1+\zeta_{5}^{-k})&2&0 \\
2(1+\zeta_{5}^{k})&2(\zeta_{5}^{k}+\zeta_{5}^{-k})&0&2 \\
0&0&0&0 \\
0&0&0&0 \\
\end{pmatrix}. 
\end{align*}
Note that $M_1=M_4$ and $M_2=M_3$. 
Moreover, the eigenvalues of $M_0$ come from the eigenvalues of the following two matrices $M_{0}', M_{0}''$, and 
the eigenvalues of $M_{k}$ come from some eigenvalues of the following two matrices $M_{k}', M_{k}''$: 
\begin{align*}
M_{0}'&=
\begin{pmatrix}
8&2 \\
10&0 \\
\end{pmatrix}, &
M_{0}''&=
\begin{pmatrix}
0&2 \\
10&0 \\
\end{pmatrix}, \\
M_{k}'&=
\begin{pmatrix}
2(\zeta_{5}^{k}+\zeta_{5}^{-k})&2(1+\zeta_{5}^{-k}) \\
2(1+\zeta_{5}^{k})&2(\zeta_{5}^{k}+\zeta_{5}^{-k}) \\
\end{pmatrix}, &
M_{k}''&=
\begin{pmatrix}
0&0 \\
0&0 \\
\end{pmatrix}. 
\end{align*}
For $M_{0}'$, the vectors
\begin{align*}
\begin{pmatrix}
1 \\
1
\end{pmatrix}, &&
\begin{pmatrix}
1 \\
-5
\end{pmatrix}
\end{align*}
are the eigenvectors associated to the eigenvalues $10$ and $-2$, respectively. 
For $M_{0}''$, the vectors
\begin{align*}
\begin{pmatrix}
1 \\
\sqrt{5}
\end{pmatrix}, &&
\begin{pmatrix}
1 \\
-\sqrt{5}
\end{pmatrix}
\end{align*}
are the eigenvectors associated to the eigenvalues $2\sqrt{5}$ and $-2\sqrt{5}$, respectively. 
For $M_{k}'$, the vectors 
\begin{align*}
\begin{pmatrix}
2+\zeta_{5}^{k}+\zeta_{5}^{-k} \\
-(1+\zeta_{5}^{k})
\end{pmatrix}, &&
\begin{pmatrix}
1+\zeta_{5}^{-k} \\
-(\zeta_{5}^{2k}+\zeta_{5}^{-2k})
\end{pmatrix}
\end{align*}
are eigenvectors associated to the eigenvalues $-2$ and $2\left((\zeta_{5}^{k}+\zeta_{5}^{-k})-(\zeta_{5}^{2k}+\zeta_{5}^{-2k})\right)$, respectively. 
For $k\in\Set{1,2,3,4}$, we have
\begin{align*}
\left((\zeta_{5}^{k}+\zeta_{5}^{-k})-(\zeta_{5}^{2k}+\zeta_{5}^{-2k})\right)^{2}=5. 
\end{align*}
If $k=1,4$, then 
\begin{align*}
(\zeta_{5}^{k}+\zeta_{5}^{-k})-(\zeta_{5}^{2k}+\zeta_{5}^{-2k})>0. 
\end{align*}
If $k=2,3$, then   
\begin{align*}
(\zeta_{5}^{k}+\zeta_{5}^{-k})-(\zeta_{5}^{2k}+\zeta_{5}^{-2k})<0.
\end{align*}
Hence the eigenvalues of $M_1=M_4$ are $-2$, $2\sqrt{5}$, and the eigenvalues of $M_2=M_3$ are $-2$, $-2\sqrt{5}$. 
For $k\in\Set{0,1,2,3,4}$, define 
\begin{align*}
\zz_{k}=
\begin{pmatrix}
1 \\
\zeta_{5}^{k} \\
\zeta_{5}^{2k} \\
\zeta_{5}^{3k} \\
\zeta_{5}^{4k} \\
\end{pmatrix}. 
\end{align*}
Then, by \cite[Lemma 2.1]{Y2018}, the vector
\begin{align*}
\begin{pmatrix}
\zz_{0} \\
\zz_{0} \\
1 \\
1
\end{pmatrix} 
\end{align*}
is a eigenvector of $H_{F_\PPP}^{1}(1,1,\ldots, 1)$ associated to the eigenvalue $10$. 
The vectors
\begin{align*}
\begin{pmatrix}
\zz_{0} \\
\zz_{0} \\
-5 \\
-5
\end{pmatrix}, 
\begin{pmatrix}
\left(2+\zeta_{5}+\zeta_{5}^{-1}\right) \zz_{1} \\
-(1+\zeta_{5})\zz_{1} \\
0 \\
0
\end{pmatrix}, 
\begin{pmatrix}
\left(2+\zeta_{5}^{2}+\zeta_{5}^{-2}\right) \zz_{2} \\
-(1+\zeta_{5}^{2})\zz_{2} \\
0 \\
0
\end{pmatrix}, \\
\begin{pmatrix}
\left(2+\zeta_{5}^{-2}+\zeta_{5}^{2}\right) \zz_{3} \\
-(1+\zeta_{5}^{-2})\zz_{3} \\
0 \\
0
\end{pmatrix}, 
\begin{pmatrix}
\left(2+\zeta_{5}^{-1}+\zeta_{5}\right) \zz_{4} \\
-(1+\zeta_{5}^{-1})\zz_{4} \\
0 \\
0
\end{pmatrix}
\end{align*}
are the eigenvectors of $H_{F_\PPP}^{1}(1,1,\ldots, 1)$ associated to the eigenvalue $-2$. 
The vectors
\begin{align*}
\begin{pmatrix}
\zz_{0} \\
-\zz_{0} \\
\sqrt{5} \\
-\sqrt{5}
\end{pmatrix}, 
\begin{pmatrix}
\left(1+\zeta_{5}^{-1}\right) \zz_{1} \\
-(\zeta_{5}^{2}+\zeta_{5}^{-2})\zz_{1} \\
0 \\
0
\end{pmatrix}, 
\begin{pmatrix}
\left(1+\zeta_{5}\right) \zz_{4} \\
-(\zeta_{5}^{-2}+\zeta_{5}^{2})\zz_{4} \\
0 \\
0
\end{pmatrix}. 
\end{align*}
are the eigenvectors of $H_{F_\PPP}^{1}(1,1,\ldots, 1)$ associated to the eigenvalue $2\sqrt{5}$. 
The vectors
\begin{align*}
\begin{pmatrix}
\zz_{0} \\
-\zz_{0} \\
-\sqrt{5} \\
\sqrt{5}
\end{pmatrix}, 
\begin{pmatrix}
\left(1+\zeta_{5}^{-2}\right) \zz_{2} \\
-(\zeta_{5}^{-1}+\zeta_{5})\zz_{2} \\
0 \\
0
\end{pmatrix}, 
\begin{pmatrix}
\left(1+\zeta_{5}^{2}\right) \zz_{3} \\
-(\zeta_{5}+\zeta_{5}^{-1})\zz_{3} \\
0 \\
0
\end{pmatrix}
\end{align*}
are the eigenvectors of $H_{F_\PPP}^{1}(1,1,\ldots, 1)$ associated to the eigenvalues $-2\sqrt{5}$.

Finally, we obtain Proposition \ref{eigen n=20}.


\bibliographystyle{amsplain-url}
\bibliography{mr}

\end{document}